\documentclass[12pt,a4paper,reqno]{amsart}

\usepackage[utf8]{inputenc}
\usepackage{eulervm}
\DeclareMathAlphabet{\mathcal}{OMS}{cmsy}{m}{n}
\usepackage{amsmath,amssymb,amsthm,amsfonts}
\usepackage{latexsym}
\usepackage{mathrsfs}
\usepackage{bbm}
\usepackage{titlesec}
\usepackage{xcolor}
\usepackage{url}
\usepackage{graphicx}
\usepackage{stmaryrd}
\SetSymbolFont{stmry}{bold}{U}{stmry}{m}{n}
\usepackage{graphicx}
\usepackage[pdfencoding=auto,colorlinks]{hyperref}

\newtheorem{thm}{Theorem}[section]
\newtheorem{cor}[thm]{Corollary}
\newtheorem{lem}[thm]{Lemma}
\newtheorem{prop}[thm]{Proposition}

\theoremstyle{definition}
\newtheorem{defn}[thm]{Definition}
\newtheorem{openprob}[thm]{Open Problem}
\newtheorem{rem}[thm]{Remark}
\newtheorem{ex}[thm]{Example}

\numberwithin{equation}{section}

\titleformat{\section}{\normalfont\bfseries\centering}{\thesection.}{.25em}{}
\titleformat{\subsection}{\normalfont\bfseries}{\thesubsection.}{.25em}{}
\titleformat{\subsubsection}{\normalfont\it}{\thesubsubsection.}{.25em}{}
\titlespacing{\section}{0pt}{*4}{*1.5}
\titlespacing{\subsection}{0pt}{*4}{*0.5}
\allowdisplaybreaks

\newcommand{\paragraf}{\textsection}
\renewcommand{\emptyset}{\varnothing}

\newcommand{\braces}[1]{{\rm (}#1{\rm )}}
\newcommand{\rmref}[1]{{\rm\ref{#1}}}
\newcommand{\ol}{\overline}

\newcommand{\wt}{\widetilde}

\newcommand{\R}{\ensuremath{\mathbb R}}    
\newcommand{\C}{\ensuremath{\mathbb C}}    
\newcommand{\N}{\ensuremath{\mathbb N}}    
\newcommand{\Z}{\ensuremath{\mathbb Z}}    

\newcommand{\calL}{\mathcal L}

\newcommand{\calX}{\mathcal X}

\newcommand{\calZ}{\mathcal Z}

\newcommand{\la}{\lambda}

\newcommand{\bmat}[4]
{
   \begin{bmatrix}
      #1 & #2\\
      #3 & #4
   \end{bmatrix}
}

\renewcommand{\Re}{\operatorname{Re}}

\newcommand{\dom}{\operatorname{dom}}

\newcommand{\ran}{\operatorname{ran}}

\newcommand{\Sra}{\Rightarrow}

\newcommand{\Llra}{\Longleftrightarrow}
\newcommand{\Slra}{\Leftrightarrow}

\DeclareMathOperator*{\wslim}{w^\ast-lim}

\newcommand{\ind}{\operatorname{ind}}

\begin{document}

\title[On DAEs with Bounded Spectrum in Banach Spaces]{On Differential-Algebraic Equations with\\Bounded Spectrum in Banach Spaces}

\author[F.~M.~Philipp]{Friedrich M.\ Philipp}
\address{{\bf F.~Philipp}
(ORCID: 0000-0002-4670-8894)
Technische Universit\"at Ilmenau, Institute of Mathematics,
Weimarer Stra\ss e 25, D-98693 Ilmenau, Germany}
\email{friedrich.philipp@tu-ilmenau.de}


\date{April 14, 2025}

\begin{abstract}
The Weierstra\ss\ form for regular DAEs in finite dimensions decouples a linear DAE into an ODE and the nilpotent part of the underlying pencil. Here, we provide necessary and sufficient conditions for the possibility of such a decomposition in the case of DAEs in Banach spaces. Moreover, we consider the larger class of linear operator pencils with bounded spectra and show that the associated homogeneous DAE can be reduced to an ODE and a seemingly simple DAE of the form $\frac d{dt}Tx = x$ with a quasi-nilpotent operator $T$. As examples show, there are cases with only the trivial solution and others with non-trivial solutions. We characterize the existence of $L^\infty$-solutions on the half-axis, $L^2$-solutions on compact time intervals, and analytic solutions.
\end{abstract}

\maketitle

\section{Introduction}
In this note we are interested in solving homogeneous differential-algebraic equations (DAEs) with an initial value condition of the form
\begin{equation}\label{e:IVP0}
\tfrac{d}{dt}Ex = Ax,\qquad Ex(0) = Ex_0,
\end{equation}
where $E$ and $A$ are linear operators on a Banach space $\calX$. We shall assume throughout that $E$ is {\em not} boundedly invertible and that the corresponding pencil $sE-A$ is regular. 

Problem \eqref{e:IVP0} can obviously be regarded as a generalized Cauchy problem. On the other hand, every Cauchy problem of the form $\dot z = Tz$, $z(0) = z_0\in\dom T$, with an operator $T$ with nonempty resolvent set $\rho(T)$ can be recast in the form \eqref{e:IVP0} by setting $E = (T-\mu)^{-1}$, $A = T(T-\mu)^{-1}$, and $x = (T-\mu)z$ for some $\mu\in\rho(T)$.

In finite dimensions, the solution theory for this problem is well understood, see, e.g., \cite{km,bit} or \cite[Appendix A]{fmpsw}. It can be shown that there always exist invertible matrices $U$ and $V$ such that $UEV$ and $UAV$ are of the form
$$
UEV = \bmat I00N\qquad\text{and}\qquad UAV = \bmat M00I,
$$
where $N$ is nilpotent and $I$ denotes the identity matrix. Obviously, this decouples the DAE in \eqref{e:IVP0} into an ordinary differential equation (ODE) and the particularly simple DAE of the form $\frac d{dt}Nz = z$, which---due to the nilpotency of $N$---has the only solution $z=0$.

The matrices $U$ and $V$ can be obtained with the help of the so-called {\em Wong sequences} (see, e.g., \cite{bit,bt,w}). These are two sequences of nested subspaces $N_k$ and $R_k$, where the $N_k$ are ascending and the $R_k$ are descending. The minimal $m\in\N$ such that $N_{m+1} = N_m$ is called the {\em index} of the problem and can be shown to coincide with the maximal length of Jordan  chains at the eigenvalue $\infty$ of the pencil $sE-A$. We then also have $R_{m+1} = R_m$ and $\calX = N_m\oplus R_m$, where $\oplus$ denotes the direct sum of subspaces. If $P_0$ denotes the projection onto $N_m$ with respect to the latter decomposition and $S = E(I-P_0) + AP_0$, then $S$ is invertible and
$$
E = S(I_{R_m}\oplus N')\qquad\text{and}\qquad A = S(M'\oplus I_{N_m})
$$
with linear maps $N' : N_m\to N_m$ and $M' : R_m\to R_m$, where $N'$ is nilpotent. The matrices $U$ and $V$ are now easy to obtain.

However, in infinite dimensions---even in the case of bounded coefficient operators $E$ and $A$---it is not clear at all whether the ascending and descending Wong sequences $N_k$ and $R_k$ become stationary at some point or whether the Wong subspaces $R_k$ are closed. One of the main results of this chapter is that the two Wong sequences become stationary if and only if the spectrum of the pencil $(E,A)$ is bounded and the resolvent $(sE-A)^{-1}$ has polynomial growth as $s\to\infty$. We say that the pencil has finite index in this case. Several equivalent conditions are provided in Theorem \ref{t:fin_ind}. The solution of the initial value problem \eqref{e:IVP0} can then be found analogously to the finite-dimensional situation, cf.\ Theorem \ref{t:ivp_solved}.

The situation is drastically different when the spectrum of $(E,A)$ is bounded but the resolvent growth at $\infty$ is {\em not} of polynomial type. In this case, the problem reduces to an ODE and a DAE of the form $\frac d{dt}Tx = x$, where the spectrum of $T$ merely consists of zero. Such operators are called quasi-nilpotent. We give two examples of quasi-nilpotent, but non-nilpotent operators $T$ in one of which the only solution to the corresponding DAE is the trivial one, but in the other non-trivial solutions also exist. We leave open the question for which quasi-nilpotent operators $T$ the solution of the DAE $\frac d{dt}Tx = x$ is unique (and thus trivial). However, we are able to characterize the existence of $L^\infty$-solutions on the half axis (Theorem \ref{t:full}), $L^2$-solutions on compact time intervals (Theorem \ref{t:l2}), and analytic solutions (Proposition \ref{p:analytic}).

The paper is organized as follows. In Section \ref{s:setting} we describe the setting of the paper and perform the first reductions: from unbounded to bounded coefficients $E$ and $A$ and from pencils $(E,A)$ to operators. We close with a generalization of the Riesz-Dunford spectral projection and show that
$$
P_\sigma := \frac 1{2\pi i}\int_C(zE-A)^{-1}E\,dz
$$
can also be regarded as a spectral projection for the pencil $(E,A)$. Here, $C$ denotes a closed curve enclosing a spectral set $\sigma$ of $(E,A)$. In Section \ref{s:ind_finite} we consider pencils of finite index and prove the main result, Theorem \ref{t:fin_ind}. After that, we provide some conditions which guarantee the closedness of the Wong subspaces $R_k$. The section ends with the above-mentioned result Theorem \ref{t:ivp_solved} on the solution of the initial value problem \eqref{e:IVP0} in the finite-index case.

Section \ref{s:qnp} is devoted to the study of the DAE $\frac d{dt}Tx = x$, where $T$ is quasi-nilpotent, but not nilpotent.  We provide the above-mentioned two examples, which show that the solution behavior of the DAE is not determined alone by the fact that $T$ is quasi-nilpotent. Here, we provide the above-mentioned necessary and sufficient conditions for the existence of $L^\infty$-solutions on $[0,\infty)$, $L^2$-solutions on a compact time interval, and analytic solutions.

\section{Setting}\label{s:setting}
In this paper, we consider initial value problems (IVP) of the form
\begin{equation}\label{e:IVP_ub}
\tfrac{d}{dt}Ex = Ax,\qquad Ex(0) = Ex_0,
\end{equation}
in a reflexive or separable Banach space $\calX$. Letting $L(\calX)$ denote the space of all bounded linear operators from $\calX$ into itself, we assume that either $E\in L(\calX)$ and $A : \calX\supset\dom A\to\calX$ is closed or $E : \calX\supset\dom E\to\calX$ is closed and $A\in L(\calX)$. We set $\dom(E,A) := \dom A$ in the first case and $\dom(E,A) := \dom E$ in the second, i.e., $\dom(E,A) = \dom E\cap\dom A$.

A {\em solution} of \eqref{e:IVP_ub} is a trajectory $x\in C(\R_0^+,\calX)$ with $x(t)\in\dom(E,A)$ for all $t\ge 0$, which satisfies $Ex\in C^1(\R_0^+,\calX)$ such that $\frac{d}{dt}(Ex)(t) = Ax(t)$ for all $t\ge 0$ and $Ex(0) = Ex_0$.

We define the {\em resolvent set} of the pencil $(E,A)$ as
$$
\rho(E,A) := \{s\in\C\,|\,sE-A : \dom (E,A)\to\calX\text{ is bijective}\}.
$$
The {\em spectrum} of $(E,A)$ is defined by $\sigma(E,A) := \C\setminus\rho(E,A)$, where we have intentionally excluded the point $\infty$.

We assume here and throughout the paper that the DAE in \eqref{e:IVP_ub} and the pencil $(E,A)$ are {\em regular}, i.e., that $\rho(E,A)\neq\emptyset$.

\subsection{Reduction to bounded coefficient operators}
We fix an arbitrary $\mu\in\rho(E,A)$ and define the bounded operators
$$
F := E(\mu E - A)^{-1}
\qquad\text{and}\qquad
B := A(\mu E - A)^{-1}.
$$
It is easy to see that if $x$ solves \eqref{e:IVP_ub}, then $z := (\mu E - A)x = \mu Ex - \frac d{dt}Ex$ is continuous and satisfies
\begin{equation}\label{e:IVP_bd}
\tfrac d{dt}Fz = Bz,\qquad Fz(0) = Fz_0,
\end{equation}
where $z_0 := (\mu E - A)x_0$. Conversely, each solution $z$ of \eqref{e:IVP_bd} transforms to the solution $x = (\mu E - A)^{-1}z$ of \eqref{e:IVP_ub}.

Therefore, we shall assume without loss of generality from Subsection \ref{ss:reduction_op} on that both operators $E$ and $A$ are bounded. We add here that the relation
$$
(sF-B)(\mu E - A) = sE-A, \qquad s\in\C,
$$
immediately yields
$$
\sigma(E,A) = \sigma(F,B).
$$
In particular, the pencil $(F,B)$ is regular with $\mu F - B = I$.

It will be shown below (cf.\ Lemma \ref{l:dae_sol}) that $\sigma(F,B)$ is bounded if and only if zero is an isolated point of the spectrum of the operator $(\mu F - B)^{-1}F = F = E(\mu E - A)^{-1}$.

\begin{ex}
{\bf (a)} If $K := A(\mu E - A)^{-1}$ is compact for some $\mu\neq 0$, we have $\mu F = I + K$, and it follows from the spectral theory for compact operators that zero is an isolated point of $\sigma(F)$. Since it is also a pole of the resolvent of $F$, the pencil $(F,B)$ has finite index, cf.\ Lemma \ref{l:easy_lem}, which is the subject of Section \ref{s:ind_finite}.

\smallskip\noindent
{\bf (b)} If $E$ is a bounded finite-rank operator, then so is $F = E(\mu E - A)^{-1}$ so that zero is an isolated spectral point of $F$.
\end{ex}

%

\subsection{Reduction to operators}\label{ss:reduction_op}
Consider the IVP \eqref{e:IVP_ub} with $E,A\in L(\calX)$. Next, we shall reduce the spectral properties of the pencil $(E,A)$ to those of a special operator. Fixing an arbitrary $\mu\in\rho(E,A)$, we apply the resolvent $(\mu E - A)^{-1}$ to both the DAE and the initial value condition and obtain the equivalent IVP
\begin{equation}\label{e:IVP_red}
\tfrac{d}{dt}T_\mu x = (\mu T_\mu - I)x,\qquad T_\mu x(0) = T_\mu x_0,
\end{equation}
where
$$
T_\mu := (\mu E-A)^{-1}E.
$$
Indeed, we have
$$
(\mu E-A)^{-1}A = (\mu E-A)^{-1}(\mu E - (\mu E - A)) = \mu T_\mu - I.
$$
Note that for $s\neq\mu$,
\begin{equation}\label{e:EAT}
sE-A = (s-\mu)(\mu E-A)\big(T_\mu - \tau_\mu(s)\big),
\end{equation}
where
$$
\tau_\mu(z) := \frac 1{\mu-z},\qquad z\in\C\setminus\{\mu\}.
$$
Therefore, for $s\neq\mu$,
\begin{equation}\label{e:corr}
s\in\sigma(E,A)\quad\Llra\quad\tau_\mu(s)\in\sigma(T_\mu).
\end{equation}
Moreover, it is easily seen that for $s,z\in\rho(E,A)$ we have
\begin{equation}\label{e:TT}
T_zT_s = -\frac{T_s-T_z}{s-z}.
\end{equation}
In particular, the operators $T_z$, $z\in\rho(E,A)$, all commute with one another.

Note that $0\in\sigma(T_\mu)$ as $E$ is assumed to be non-invertible. What we now would like to assume is that {\em zero is an isolated point of the spectrum of $T_\mu$}. In this case, we may apply the Riesz-Dunford spectral projector\footnote{where $C = \partial B_r(0)$, $r>0$ sufficiently small, is a positively oriented circle line around the zero point}
\begin{equation}\label{e:rdp}
P_0 = -\frac 1{2\pi i}\int_{C}(T_\mu - \la)^{-1}\,d\la
\end{equation}
and the complementary projection $I-P_0$ to \eqref{e:IVP_red}, respectively, to obtain the two decoupled IVPs:
\begin{align}
\begin{split}\label{e:IVP_proj}
\tfrac{d}{dt}S_0z_0 &= (\mu S_0 - I)z_0,\qquad S_0z_0(0) = S_0P_0x_0,\\
\tfrac{d}{dt}S_1z_1 &= (\mu S_1 - I)z_1,\qquad S_1z_1(0) = S_1(I-P_0)x_0,
\end{split}
\end{align}
where $S_0 = T_\mu|_{P_0\calX}\in L(P_0\calX)$ and $S_1 = T_\mu|_{(I-P_0)\calX}\in L((I-P_0)\calX)$. Then $S_1$ is boundedly invertible and $\sigma(S_0) = \{0\}$. Hence, the set of IVPs \eqref{e:IVP_proj} is equivalent to
\begin{align}
\tfrac{d}{dt}Tz_0 &= z_0,\qquad\qquad\quad\;\; Tz_0(0) = TP_0x_0,\label{e:IVP_qnp}\\
\dot z_1 &= (\mu - S_1^{-1})z_1,\qquad z_1(0) = (I-P_0)x_0,\label{e:IVP_inv}
\end{align}
where $T = (\mu S_0 - I)^{-1}S_0$. Note that also $\sigma(T) = \{0\}$.

Clearly, the IVP \eqref{e:IVP_inv} has the unique solution $z_1(t) = e^{(\mu-S_1^{-1})t}(I-P_0)x_0$ so that it remains to solve the IVP \eqref{e:IVP_qnp}. If $z_0$ is a solution, then the original problem \eqref{e:IVP_red} has the solution $x = z_0 + z_1$.

The next lemma immediately follows from \eqref{e:corr}.

\begin{lem}\label{l:dae_sol}
Zero is an isolated point of $\sigma(T_\mu)$ if and only if $\sigma(E,A)$ is bounded. In this case, the IVP \eqref{e:IVP_ub} has a solution if and only if the IVP \eqref{e:IVP_qnp} has a solution.
\end{lem}

\subsection{Spectral sets}
A set $\sigma\subset\sigma(E,A)$ will be called a {\em spectral set} for the pencil $(E,A)$, if it is both open and closed in $\sigma(E,A)$. If $\sigma$ is a bounded spectral set for $(E,A)$, then we find a bounded open neighborhood $U$ of $\sigma$ whose boundary $C$ consists of a finite number of rectifiable Jordan curves, oriented in the positive sense, and such that $\ol U\cap\sigma(E,A) = \sigma$. We then define
$$
P_\sigma := \frac 1{2\pi i}\int_C(zE-A)^{-1}E\,dz = \frac 1{2\pi i}\int_CT_z\,dz.
$$
It is evident that this definition does not depend on the set $U$ and its boundary, but only on $\sigma$.

Parts of the following proposition have been proved in \cite[Theorem 5.1]{mga} for the special case of isolated points of $\sigma(E,A)$, i.e., $\sigma = \{\la_0\}$ with $\la_0\in\C$. Its proof can be found in the appendix of this paper.

\begin{prop}\label{p:proj}
Let $\sigma$ be a bounded spectral set of $(E,A)$, let $\mu\in\rho(E,A)$ be arbitrary, and choose $C$ as above such that $\mu$ is in the exterior of $C$. Then zero is in the exterior of $\tau_\mu(C)$, $\tau_\mu(C)$ is positively oriented, and $\tau_\mu(\sigma)$ is a spectral set for the operator $T_\mu$ which is contained in the interior of $\tau_\mu(C)$. Moreover,
\begin{equation}\label{e:Psigma}
P_\sigma = -\frac 1{2\pi i}\,\int_{\tau_\mu(C)}(T_\mu - \la)^{-1}\,d\la,
\end{equation}
i.e., $P_\sigma$ is the spectral projection of $T_\mu$ corresponding to the spectral set $\tau_\mu(\sigma)$.
\end{prop}

\section{DAEs of finite index}\label{s:ind_finite}
Next, we carry over the notion of the index of a DAE or pencil from the finite-dimensional case to the infinite-dimensional situation.

\begin{defn}
The regular DAE $\frac d{dt}Ex = Ax$ (or the regular pencil $(E,A)$) is said to have {\em finite index} $m\in\N$ if $\sigma(E,A)$ is bounded and $\|(sE - A)^{-1}\| = O(|s|^{m-1})$ as $|s|\to\infty$, where $m$ is minimal, i.e., $\|(sE - A)^{-1}\|\neq O(|s|^{m-2})$ as $|s|\to\infty$.
\end{defn}

Recall the operators $T_\mu = (\mu E-A)^{-1}E$, where $\mu\in\rho(E,A)$. The next lemma is an immediate consequence of \eqref{e:EAT}.

\begin{lem}\label{l:easy_lem}
The regular pencil $(E,A)$ has finite index $m$ if and only if zero is a pole of the resolvent of $T_\mu$ of order $m$.
\end{lem}

Let $T : \calX\to\calX$ be a linear operator. Note that for all $k\in\N$ we have $\ker T^k\subset\ker T^{k+1}$ and that $\ker T^k = \ker T^{k+1}$ implies that $\ker T^{k+j}=\ker T^k$ for all $j\in\N$. Similarly, $\ran T^{k+1}\subset\ran T^k$ and equality implies $\ran T^{k+j} = \ran T^k$ for all $j\in\N$. Hence, the closed subspaces $\ker T^k$ are ascending and the subspaces $\ran T^k$ are descending. The {\em ascent} and the {\em descent} of $T$ are defined as
\begin{align*}
\alpha(T) &:= \min\{n\in\N : \ker T^{n+1} = \ker T^n\}
\quad\text{and}\\
\delta(T) &:= \min\{n\in\N : \ran T^{n+1} = \ran T^n\},
\end{align*}
respectively. If the minimum does not exist, we put $\alpha(T) = \infty$ or $\delta(T) = \infty$, accordingly.

\begin{thm}[{\cite[Corollary 20.5]{m}}]\label{t:mueller}
Let $T\in L(\calX)$ such that $\alpha(T),\,\delta(T)<\infty$. Then $\alpha(T) = \delta(T) =: n$, $\ran T^n$ is closed, and $\calX = \ker T^n\oplus\ran T^n$.
\end{thm}

The next proposition shows in particular that the condition in Theorem \ref{t:mueller} means that zero is a pole of the resolvent of $T$. In the proof we shall frequently make use of the relation
\begin{equation}\label{e:geometric}
(T-\la)^{-1}(T^{n}-\la^{n}) = \sum_{k=0}^{n-1}\la^kT^{n-k-1},
\end{equation}
which holds for $T\in L(\calX)$, $n\in\N_{\ge 1}$, and $\la\in\rho(T)$.

\begin{prop}\label{p:fin_ind}
For $T\in L(\calX)$ and $m\in\N_{\ge 1}$ the following statements are equivalent:
\begin{enumerate}
\item[{\rm (i)}]   Zero is a pole of the resolvent of $T$ of order $m$.
\item[{\rm (ii)}]  Zero is an isolated point of $\sigma(T)$ and $T^m\int_C(T-\la)^{-1}\,d\la = 0$ \braces{$m$ minimal\,}, where $C=\partial B_r(0)$ such that $\ol{B_r(0)}\cap\sigma(T) = \{0\}$.
\item[{\rm (iii)}] Zero is an isolated point of $\sigma(T)$ and $\delta(T) = m$.
\item[{\rm (iv)}]  $\calX = \ker T^m\oplus \ran T^m$ \braces{$m$ minimal\,}.
\item[{\rm (v)}]   $\alpha(T) = \delta (T) = m$.
\item[{\rm (vi)}]  There exists a sequence $(\la_n)\subset\rho(T)$ such that $\la_n\to 0$ and
\begin{align}\label{e:seq}
\|(T-\la_n)^{-1}\| = O(|\la_n|^{-m})\,\text{ as }\,n\to\infty\quad\text{\braces{$m$ minimal\,}},
\end{align}
and $\ran T^m$ is closed.
\item[{\rm (vii)}] There exists a sequence $(\la_n)\subset\rho(T)$ with $\la_n\to 0$ as in \eqref{e:seq}, and $\delta(T) < \infty$.
\end{enumerate}
\end{prop}
\begin{proof}
Assume that all items except (iii) are equivalent. Then (ii) and (v) certainly imply (iii). Conversely, if (iii) holds, denote the spectral projection for $T$ with respect to the spectral set $\sigma = \{0\}$ by $P_0$. Then $T_0:=T|_{P_0\calX}$ is quasi-nilpotent (i.e., $\sigma(T_0)=\{0\}$) and $\delta(T_0) = m$. Hence, \cite[Theorem 2]{g} implies that $T_0$ is nilpotent. Consequently, we have $\alpha(T) = \delta(T) = m$ and thus (v). This shows that we may neglect (iii) in the following.

\smallskip\noindent
(i)$\Slra$(ii) This is the statement of \cite[Theorem VII.3.18]{ds}.

\smallskip\noindent
(ii)$\Sra$(iv) Let $P_0 = -\frac 1{2\pi i}\int_C(T-\la)^{-1}\,d\la$ be the spectral projection of $T$ with respect to the spectral set $\{0\}$. Then $T^mP_0=0$ means that $P_0\calX = \ker T^m$. Hence $\calX = \ker T^m\oplus (I-P_0)\calX$ implies that $\ran T^m = T^m(I-P_0)\calX = (I-P_0)\calX$ as $T|_{(I-P_0)\calX}$ is boundedly invertible. As to the minimality of $m$, assume that $\calX = \ker T^k\oplus\ran T^k$ for some $1\le k\le m$. Then $\ran T^k\subset\ran T^{2k}$ and hence $\ran T^k\subset\ran T^{jk}$ for all $j\in\N$. Choose $j\in\N$ such that $jk\ge m$. Then from $T^kP_0\calX = P_0T^k\calX\subset P_0T^{jk}\calX = T^{jk}P_0\calX = \{0\}$ and the minimality in (ii) we obtain $k=m$.

\smallskip\noindent
(iv)$\Sra$(v) Let $x\in\ker T^{m+1}$. Then also $x\in\ker T^{2m}$ so that $T^mx\in\ker T^m\cap\ran T^m=\{0\}$, i.e., $x\in\ker T^m$. It follows that $\alpha(T)\le m$. If $y\in\ran T^m$, $y = T^mx$ with some $x\in\calX$, then $x = u + T^mz$ with some $u\in\ker T^m$ and $z\in\calX$. Thus, $y = T^{2m}z\in\ran T^{m+1}$. This implies $\delta(T)\le m$. Now, (v) follows from Theorem \ref{t:mueller} and the minimality in (iv).

\smallskip\noindent
(v)$\Sra$(i) By Theorem \ref{t:mueller}, $\ran T^m$ is closed and $\calX = \ker T^m\oplus\ran T^m$. With respect to this decomposition, $T$ decomposes as $T = N\oplus S$, where $N\in L(\ker T^m)$ is nilpotent ($N^m=0$, $N^{m-1}\neq 0$) and $S\in L(\ran T^m)$ is boundedly invertible. In particular, $\sigma(T) = \sigma(N)\cup\sigma(S) = \{0\}\cup\sigma(S)$, which shows that zero is an isolated point of $\sigma(T)$. Now, (i) follows from the representation $(N-\la)^{-1} = -\sum_{j=-m}^{-1}\la^jN^{-j-1}$ of the resolvent of the nilpotent operator $N$, which can be deduced from \eqref{e:geometric}.

\smallskip
Before we turn to the proof of (i)$\Sra$(vi)$\Sra$(vii)$\Sra$(i), let us show that the existence of a sequence $(\la_n)$ as in (vi) and (vii) implies $\alpha(T)\le m$. Indeed, \eqref{e:geometric} with $n=m+1$ yields for $\la\in\rho(T)$ that
$$
(T-\la)^{-1}(T^{m+1}-\la^{m+1}) = \sum_{k=0}^{m}\la^kT^{m-k}.
$$
Applying this to $x\in\ker T^{m+1}$ with $\la = \la_n$ and letting $n\to\infty$ in fact yields $T^mx=0$, i.e., $x\in\ker T^m$.

\smallskip\noindent
(i)$\Sra$(vi) We can choose any sequence $(\la_n)$ such that $\la_n\to 0$ to satisfy the requirements in (vi). The closedness of $\ran T^m$ follows from (i)$\Slra$(iv) and Theorem \ref{t:mueller}.

\smallskip\noindent
(vi)$\Sra$(vii) The following argumentation is valid if $\calX$ is separable: Since $\ran T^m$ is closed, so is $\ran(T')^m$. Here, $T' : \calX'\to\calX'$ denotes the adjoint of $T$. Note that $\sigma(T') = \sigma(T)$ and $\|(T'-\la)^{-1}\| = \|(T-\la)^{-1}\|$ for $\la\in\rho(T)$. Hence, by the intermediate part of this proof, $\alpha(T')\le m$. Hence, we only have to prove that $\ran(T')^{m+1} = \ran(T')^m$, since then $\ran(T')^{m+1}$ is closed and consequently $\ran T^{m+1} = {}^\perp\!\ker(T')^{m+1} = {}^\perp\!\ker(T')^{m} = \ran T^m$. For this, let $x_j'\in\ran(T')^j$ for $j=0,\ldots,m$ such that $T'x_j' = x_{j+1}'$, $j=0,\ldots,m-1$. It follows from \eqref{e:geometric} that for $\la\in\rho(T')$ we have
$$
(T'-\la)^{-1}x_m' = \sum_{j=0}^{m-1}\la^jx_{m-1-j}' + \la^m(T'-\la)^{-1}x_0'.
$$
Therefore, due to the property of $(\la_n)$, the sequence $(u_n')$ with $u_n' = (T'-\la_n)^{-1}x_m'\in\ran(T')^m$ is bounded. By the sequential Banach-Alaoglu theorem (see, e.g., \cite[Theorem 3.17]{r}) $(u_n')$ has a weak$^*$-convergent subsequence. Without loss of generality, we may assume that this subsequence is $(u_n')$ itself. Set $u' := \wslim_{n\to\infty}u_n'$. As $\ran(T')^m$ is weak${}^*$-closed (see \cite[Theorem 4.14]{r}), we have $u'\in\ran(T')^m$. Moreover,
$$
T'u' = \wslim_{n\to\infty}\,T'(T'-\la_n)^{-1}x_m' = x_m' + \wslim_{n\to\infty}\,\la_nu_n' = x_m',
$$
and thus $x_m'\in\ran(T')^{m+1}$.

If $\calX$ is reflexive, we can prove $\ran T^{m+1} = \ran T^m$ by proceeding similarly as in the above proof with $T$ instead of $T'$ and by using the fact that bounded sequences in $\ran T$ have a weakly convergent subsequence with weak limit in $\ran T$.

\smallskip\noindent
(vii)$\Sra$(i) By the intermediate part of the proof and Theorem \ref{t:mueller}, we have that $k := \alpha(T) = \delta(T)\le m$. Hence, due to the already proved equivalence of (v) and (i), it follows that zero is a pole of the resolvent of $T$ of order $k$. But $m$ in (vii) is minimal so that $k=m$.
\end{proof}

The so-called {\em Wong sequences} $(N_k)$ and $(R_k)$ (see, e.g., \cite{bit,bt,bt2,w}) for the pencil $(E,A)$ are defined by $N_0 = \{0\}$, $R_0 = \calX$, and
$$
N_{k+1} = E^{-1}AN_k\qquad\text{and}\qquad R_{k+1} = A^{-1}ER_k
$$
for $k=0,1,\dots$. It is easily seen that the $N_k$ are ascending and the $R_k$ are descending.

The statement of the next lemma holds in full generality---even for singular pencils $(E,A)$. It was proved as \cite[Lemma 2.3]{bit} for the special case of finite-dimensional regular pencils and only for $s\in\rho(E,A)$.

\begin{lem}\label{l:translation}
For each $s\in\C$ we have
$$
R_{k+1} = (sE-A)^{-1}ER_k
\qquad\text{and}\qquad
N_{k+1} = E^{-1}(sE-A)N_k.
$$
\end{lem}
\begin{proof}
Let $s\in\C$. If $x\in R_{k+1}$, then $Ax\in ER_k$ and also $x\in R_k$. Hence, we have $(sE-A)x = sEx - Ax\in ER_k$, which gives $x\in (sE-A)^{-1}ER_k$. For the opposite direction, let $x\in (sE-A)^{-1}ER_k$, i.e., $(sE-A)x\in ER_k\subset ER_\ell$ for all $\ell = 0,\ldots,k$. From $(sE-A)x\in ER_0$ we conclude that $Ax\in ER_0$ and thus $x\in R_1$. Now, an inductive argument leads to $x\in R_k$ and thus $Ax\in ER_k$, which means $x\in R_{k+1}$.

If $x\in E^{-1}(sE-A)N_k$, then $Ex = (sE-A)u$ with some $u\in N_k$. Hence, $E(su-x) = Au\in AN_k$, i.e., $su-x\in N_{k+1}$. But also $u\in N_{k+1}$ and so $x\in N_{k+1}$. The converse inclusion is proved via induction. Clearly, $N_1 = \ker E = E^{-1}(sE-A)N_0$. Now, if $N_{k+1}\subset E^{-1}(sE-A)N_k$ holds for some $k$, let $x\in N_{k+2}$. Then $Ex = Au$ with $u\in N_{k+1}$, thus $Eu = (sE-A)v$ with $v\in N_k\subset N_{k+1}$, and we obtain
$$
(sE-A)(sv-u) = s(sE-A)v - sEu + Au = Au = Ex.
$$
Since $sv-u\in N_{k+1}$, it follows that $x\in E^{-1}(sE-A)N_{k+1}$, as desired.
\end{proof}

The next corollary follows directly from the definitions of the $R_k$ and $N_k$ and the fact that $\ker ST = T^{-1}\ker S$ for all $S,T\in L(\calX)$.

\begin{cor}\label{c:ranker}
For any $\mu\in\rho(E,A)$ and $k\in\N_0$ we have
$$
R_k = \ran T_\mu^k
\qquad\text{and}\qquad
N_k = \ker T_\mu^k.
$$
\end{cor}

Combining Corollary \ref{c:ranker} and Proposition \ref{p:fin_ind} immediately yields the following theorem characterizing the finite-index property of the regular pencil $(E,A)$.

\begin{thm}\label{t:fin_ind}
Let $m\in\N$. Then the following statements are equivalent:
\begin{enumerate}
\item[{\rm (i)}]   The pencil $(E,A)$ has finite index $m$.
\item[{\rm (ii)}]  $\sigma(E,A)$ is bounded and $R_{m+1} = R_m$ \braces{$m$ minimal\,}.
\item[{\rm (iii)}] $\calX = N_m\oplus R_m$ \braces{$m$ minimal\,}.
\item[{\rm (iv)}]  $R_{m+1} = R_m$ and $N_{m+1} = N_m$ \braces{$m$ minimal\,}.
\item[{\rm (v)}]   There exists a sequence $(s_n)\subset\rho(E,A)$ such that $|s_n|\to\infty$ and
\begin{align}\label{e:sequence}
\|(s_nE-A)^{-1}\| = O(|s_n|^{m-1})\,\text{ as }\,n\to\infty\quad\text{\braces{$m$ minimal\,}},
\end{align}
and $R_m$ is closed.
\item[{\rm (vi)}]  There exists a sequence $(s_n)\subset\rho(E,A)$ with $|s_n|\to\infty$ as in \eqref{e:sequence}, and $R_{k+1} = R_k$ for some $k\in\N$.
\end{enumerate}
\end{thm}

\begin{cor}
The pencil $(E,A)$ has finite index if and only if $R_{k+1} = R_k$ and $N_{j+1} = N_j$ for some $k,j\in\N$.
\end{cor}

\begin{rem}
In \cite[Problem 1.2]{tw} it was asked as to whether the closedness of all $R_k$ and the index of $(E,A)$ being $m$ imply $\min\{k : R_k = R_{k+1}\} = m$. The implication (v)$\Sra$(iv) in Theorem \ref{t:fin_ind} answers this question in the affirmative.
\end{rem}

The following example shows that the Wong subspaces $R_1,\ldots,R_{m-1}$ do not necessarily have to be closed.

\begin{ex}
Assume that the equivalent conditions in Theorem \ref{t:fin_ind} are satisfied and let $\mu\in\rho(E,A)$. Then $T_\mu$ decomposes as $T_\mu = N\oplus S$ with respect to the decomposition $\calX = N_m\oplus R_m$, where $N$ is a nilpotent operator and $S$ is boundedly invertible. Hence,
$$
R_k = \ran T_\mu^k = \ran N^k\oplus\ran S^k = \ran N^k\oplus R_m
$$
is closed if and only if $\ran N^k$ is closed. The question therefore is whether there exist nilpotent operators with non-closed range. And in fact there are. As an example consider any bounded operator $S$ with non-closed range on a Banach space $\calZ$. Then $N : \calZ^2\to\calZ^2$,
$$
N := \bmat 0S00
$$
satisfies $N^2=0$ and $\ran N = \ran S\times\{0\}$ is not closed.
\end{ex}

Recall that an operator $T\in L(\calX)$ is called {\em upper semi-Fredholm} if $\dim\ker T < \infty$ and $\ran T$ is closed in $\calX$. The operator $T$ is called {\em Fredholm} if it is upper semi-Fredholm and, in addition, $\dim(\calX/\ran T) < \infty$. The {\em Fredhom index} of a Fredholm operator $T\in L(\calX)$ is defined by
$$
\ind(T) = \dim\ker T - \dim(\calX/\ran T).
$$
Next, we present two sufficient conditions for the spaces $R_k$ to be closed.

\begin{lem}\label{l:Rclosed}
The Wong subspaces $R_k$ are closed in the following cases:
\begin{enumerate}
\item[\rm (a)] $E$ is upper semi-Fredholm.
\item[\rm (b)] $\mu E-A$ is compact for some $\mu\in\C$.
\end{enumerate}
\end{lem}
\begin{proof}
(a) It is well known (see \cite[Lemma IV.5.29]{k}) that an upper semi-Fredholm operator maps closed subspaces to closed subspaces. Since $A$ is continuous, the claim follows right from the definition of the $R_k$.

(b) Let $A_0 := \mu E-A$ be compact. By Lemma \ref{l:translation} the $R_k$ do not change if we pass in their definition from $(E,A)$ to $(E,A_0)$. Hence, the same lemma gives
$$
R_{k+1} = (sE-A_0)^{-1}ER_k = \big(I + (sE-A_0)^{-1}A_0\big)R_k
$$
for $s\in\rho(E,A_0)\setminus\{0\}$. Since the operator $I + (sE-A_0)^{-1}A_0$ is Fredholm (with Fredholm index zero), see \cite[Theorem IV.5.26]{k}, all $R_k$ are closed.
\end{proof}


\begin{cor}
Assume $\dim\ker E < \infty$. Then the pencil $(E,A)$ has finite index at most $m$ if and only if $E$ is Fredholm with Fredholm index zero and $\ker E\cap R_m = \{0\}$.
\end{cor}
\begin{proof}
Let $\mu\in\rho(E,A)$. Note that $\dim\ker E < \infty$ implies $\dim\ker T_\mu < \infty$ and hence also $\dim\ker T_\mu^m < \infty$.

Assume that $(E,A)$ has finite index at most $m$. Then, by Theorem \ref{t:fin_ind}, $R_m = \ran T_\mu^m$ is closed, and we have $\calX = \ker T_\mu^m\oplus\ran T_\mu^m$. That is, $T_\mu^m$ is Fredholm with $\ind(T_\mu^m) = 0$. By \cite[Theorem 16.6]{m}, $T_\mu$ is Fredholm and the index theorem \cite[Theorem 16.12]{m} yields $\ind(T_\mu^m) = m\cdot\ind(T_\mu)$ so that $\ind(T_\mu)=0$, and hence $E$ is Fredholm with $\ind(E) = \ind((\mu E - A)^{-1}) + \ind(E) = \ind((\mu E - A)^{-1}E) = \ind(T_\mu) = 0$. Also, $\ker E\cap R_m = N_1\cap R_m\subset N_m\cap R_m = \{0\}$.

Conversely, assume that $E$ is Fredholm with $\ind(E) = 0$ and $\ker E\cap R_m = \{0\}$. The latter is equivalent to $\ker T_\mu\cap\ran T_\mu^m = \{0\}$ and means that $\ker T_\mu^m = \ker T_\mu^{m+1}$, i.e., $\alpha(T_\mu)\le m$. Moreover, it is evident that also $T_\mu$ is Fredholm with $\ind(T_\mu)=0$ so that from \cite[Thm. 4.3]{kaa} we infer that $\delta(T_\mu)=\alpha(T_\mu)\le m$. Now, it follows from Proposition \ref{p:fin_ind} that $(E,A)$ has finite index at most $m$.
\end{proof}

We now apply the results to DAEs
\begin{equation}\label{e:ivp}
\tfrac d{dt}Ex = Ax,\qquad Ex(0) = Ex_0.
\end{equation}

\begin{thm}\label{t:ivp_solved}
Assume that the regular pencil $(E,A)$ has finite index $m$. Then the IVP \eqref{e:ivp} has a solution if and only if $x_0\in R_m\oplus\ker E$, in which case the solution is unique.
\end{thm}
\begin{proof}
Choose some $\mu\in\rho(E,A)$ and define $P_0$ as in \eqref{e:rdp}. Then $P_0\calX = N_m = \ker T_\mu^m$ and $(I-P_0)\calX = R_m = \ran T_\mu^m$. Hence, $S_0 = T_\mu|_{P_0\calX}$ is nilpotent and hence so is $T = (\mu S_0 - I)^{-1}S_0$. By Lemma \ref{l:dae_sol}, the IVP \eqref{e:ivp} has a solution if and only if \eqref{e:IVP_qnp} has a solution. In fact, the DAE in \eqref{e:IVP_qnp} has only the trivial solution. Namely, if $z_0$ is a solution, then $0 = \frac d{dt}T^m z_0 = T^{m-1}z_0$ and thus $z_0(t)\in\ker T^{m-1}$ for all $t\ge 0$. Further processing this argument leads to $z_0(t)\in\ker T$ for all $t\ge 0$ and hence finally $z_0 = \frac d{dt}Tz_0 = 0$. Consequently, the initial value condition in \eqref{e:IVP_qnp} is feasible if and only if $TP_0x_0 = 0$, which means that $x_0\in R_m\oplus\ker E$ as $\ker T = \ker E$.
\end{proof}

\begin{cor}
If the regular pencil $(E,A)$ has index $m=1$, then \eqref{e:ivp} has a unique solution for all $x_0\in\calX$.
\end{cor}
\begin{proof}
By Theorem \ref{t:ivp_solved}, the IVP \eqref{e:ivp} has a solution for any $x_0\in R_1\oplus\ker E$. But $\ker E = N_1$ and $R_1\oplus N_1 = \calX$.
\end{proof}

\section{DAEs with bounded spectrum and infinite index}\label{s:qnp}
In this section, we assume that the spectrum $\sigma(E,A)$ of the pencil $(E,A)$ is bounded. By Lemma \ref{l:dae_sol}, the IVP \eqref{e:IVP_ub} then has a solution if and only if \eqref{e:IVP_qnp} has a solution. Therefore, we shall now focus on DAEs as in \eqref{e:IVP_qnp} of the form
\begin{equation}\label{e:qn}
\tfrac d{dt}Tx = x,
\end{equation}
where $\sigma(T)=\{0\}$. Such operators are called {\em quasi-nilpotent}. A quasi-nilpo\-tent operator $T$ is nilpotent if and only if zero is a pole of its resolvent (cf.\ Proposition \ref{p:fin_ind}). In this case, the DAE \eqref{e:qn} only has the trivial solution $x=0$ (see the proof of Theorem \ref{t:ivp_solved}). Otherwise, $T$ is not nilpotent and zero is an essential singularity of the resolvent of $T$. In this case, the question arises whether solutions of \eqref{e:qn} are still necessarily trivial. The next proposition shows that this is not the case, in general.

\begin{prop}\label{p:C0s}
Let $A$ be a generator of a $C_0$-semigroup $(T_t)_{t\ge 0}$ of operators $T_t\in L(\calX)$ with $\sigma(A)=\emptyset$. Then $T=A^{-1}$ is quasi-nilpotent and for any $x_0\in\calX$ the IVP $\tfrac d{dt}Tx = x$, $x(0)=x_0$, has the unique solution $T_tx_0$.
\end{prop}
\begin{proof}
First of all, it is easy to see that $T$ is quasi-nilpotent. Let $x_0\in\calX$ be arbitrary. In fact, the proposed solution $\bar x(t) := T_tx_0$ is evidently continuous as a function from $\R_0^+$ to $\calX$ and, since $y_0 := A^{-1}x_0\in\dom A$ and $T_ty_0$ solves the Cauchy problem $\dot y = Ay$, $y(0) = y_0$, we have
$$
\tfrac d{dt}T\bar x(t) = \tfrac d{dt}A^{-1}T_tx_0 = \tfrac d{dt}T_tA^{-1}x_0 = AT_tA^{-1}x_0 = T_tx_0 = \bar x(t).
$$
Hence, $\bar x$ is a solution of the IVP. If $\wt x$ is another solution, set $z := \bar x - \wt x$ and $w = A^{-1}z$. Then $\frac d{dt}Tz = z$ and $z(0)=0$. In particular, $w = Tz\in C^1(\R_0^+,\calX)$ and $\dot w = z = Aw$, $w(0) = 0$. Consequently , $w=0$ and so $z=Aw=0$, i.e., $\wt x = \bar x$.
\end{proof}

Hence, if $T^{-1}$ is the generator of a $C_0$-semigroup, the DAE $\frac d{dt}Tx = x$ behaves like an infinite-dimensional well-posed ODE---in a sense even better, because it has a unique ``classical'' solution for every initial value $x_0\in\calX$.

\begin{ex}\label{ex:malsomalso}
Consider the Volterra operator $V : L^2(0,1)\to L^2(0,1)$, defined by
$$
(Vu)(s) = \int_0^su(\sigma)\,d\sigma.
$$
It is well known that $V$ is quasi-nilpotent and compact.

\smallskip\noindent
{\bf (a)} The operator $T = -V$ is the inverse of the differential operator $A : L^2(0,1)\supset\dom A\to L^2(0,1)$ with $Af = -f'$ for $f$ in
$$
\dom A = \{u\in W^{1,2}(0,1) : u(0) = 0\}.
$$
It is well known that the operator $A$ is the generator of the $C_0$-semigroup of contractions $(T_t)_{t\ge 0}$, given by
\[
T_tu(s) = \begin{cases}u(s-t) &\text{for $t\le s\le 1$}\\0 &\text{otherwise}\end{cases}.
\]
Hence, the DAE $\frac d{dt}Tx = x$ has non-trivial solutions by Proposition \ref{p:C0s}.

\smallskip
\noindent{\bf (b)} The DAE $\frac d{dt}Tx = x$ with $T=V$ only has the trivial solution on $[0,\infty)$. Indeed, if $x$ is a solution on $[0,\infty)$ and $\tau>1$, then $z(t) = x(\tau-t)$ for $t\in [0,\tau]$, solves $-\frac d{dt}Tz = z$ on $[0,\tau]$ and is therefore given by $z(t) = T_tz(0)$. Hence, $x(t) = T_{\tau-t}x(\tau) = 0$ for $t<\tau-1$. Since $\tau>1$ was arbitrary, the claim follows. We further would like to point out that the function $x(t) = T_{\tau-t}x(\tau)$ solves the IVP $\frac d{dt}Tx = x$, $x(0) = 0$, on $[0,\tau]$, but does not vanish.

\end{ex}

In view of Example \ref{ex:malsomalso}, the question arises for which quasi-nilpotent operators the DAE \eqref{e:qn} only has the trivial solution on $[0,\infty)$. We pose this as an open problem.

\begin{openprob}
For which quasi-nilpotent operators $T\in L(\calX)$ does the DAE $\frac d{dt}Tx = x$ only have the trivial solution $x=0$ on $[0,\infty)$?
\end{openprob}

\begin{rem}
Note that if $x$ is a solution of \eqref{e:qn}, then $T^nx$, $n\in\N$, is a solution. Thus, for any function $f$ analytic on a neighborhood of zero we have that $f(T)x$ is a solution.
\end{rem}

In the following three subsections, we collect some necessary and sufficient conditions for the existence of solutions of the DAE $\frac d{dt}Tx=x$.

\subsection{Essentially bounded solutions on the half-axis}
In this subsection, we study $L^\infty$-solutions of the IVP
\begin{align}\label{e:IVP99}
\tfrac d{dt}Tx = x,\quad Tx(0) = Tx_0,
\end{align}
on $[0,\infty)$. For this, we assume that the Banach space $\calX$ has the Radon-Nikodym property. It is well known that reflexive Banach spaces have this property. W.\ Arendt proved in \cite{a} that $\calX$ has the Radon-Nikodym property if and only if each Lipschitz-continuous function $F : [0,1]\to\calX$ is a.e.\ differentiable.

\begin{thm}\label{t:full}
Let $\calX$ have the Radon-Nikodym property. Then the IVP \eqref{e:IVP99} has a solution $x\in L^\infty([0,\infty),\calX)$ if and only if
\begin{align}\label{e:cond99}
\sup\big\{\big\|[(T-s)^{-1}T]^{n+1}x_0\big\| : s>0,\,n\in\N_0\big\} < \infty.
\end{align}
The solution then is unique.
\end{thm}
\begin{proof}
In what follows, we make use of the Laplace transform
\[
\calL z(s) = \int_0^\infty e^{-st}z(t)\,dt,\quad s>0,
\]
which is well defined for measurable functions $z : [0,\infty)\to\calX$ with polynomial growth, i.e., $\|z(t)\|\le M(1+t)^n$ for some $n\in\N$ and $M>0$ and all $t\ge 0$, and defines an analytic $\calX$-valued function on $\C^+ = \{\la\in\C : \Re\la>0\}$. It is easy to see that the Laplace transform $Z$ of a function $z\in L^\infty([0,\infty),\calX)$ satisfies
\begin{align}\label{e:absch}
\tfrac 1{n!}\big\|s^{n+1}Z^{(n)}(s)\big\|\,\le\,\|z\|_\infty,\quad n\in\N_0,\,s>0.
\end{align}
Let $x\in L^\infty([0,\infty),\calX)$ be a solution of \eqref{e:IVP99}. Applying the Laplace transform to $\frac d{dt}Tx = x$ gives $-Tx_0 + sT\calL x(s) = \calL x(s)$ and thus $X(s) := \calL x(s) = (sT-I)^{-1}Tx_0$. By \eqref{e:absch} for $n\in\N_0$ and $s>0$ we have
\begin{align*}
\big\|\big[(T-\tfrac 1s)^{-1}T\big]^{n+1}x_0\big\|
&= \big\|s^{n+1}(sT-I)^{-(n+1)}T^{n+1}x_0\big\|\\
&= \tfrac 1{n!}\big\|s^{n+1}X^{(n)}(s)\big\|\,\le\,\|x\|_\infty,
\end{align*}
which implies \eqref{e:cond99}.

Conversely, assume that \eqref{e:cond99} holds and define the function $X(s) = (sT-I)^{-1}Tx_0$, $s>0$. Then \eqref{e:cond99} implies that
\[
\sup\big\{\tfrac 1{n!}\big\|s^{n+1}X^{(n)}(s)\big\| : s>0,\,n\in\N_0\big\} < \infty.
\]
By \cite[Theorem 1.4]{a} this implies that there exists a unique $x\in L^\infty([0,\infty),\calX)$ such that $X = \calL x$. Hence, we have $(sT-I)^{-1}Tx_0 = \calL x(s)$ and thus $Tx_0 = (sT-I)\calL x(s)$, i.e.,
\[
\calL[Tx_0](s) = \tfrac 1s Tx_0 = \calL[Tx](s) - \tfrac 1s\calL x(s).
\]
Define $y(t) := \int_0^t x(\tau)\,d\tau$. Then $y$ has polynomial growth and $\calL y(s) = \tfrac 1s\calL x(s)$, which gives $\calL[Tx-Tx_0 - y]=0$. By the uniqueness theorem for the Laplace transform (see, e.g., \cite[\paragraf 5 Corollary 7.2]{wi}), we obtain $Tx-Tx_0 = y$, which is equivalent to \eqref{e:IVP99}.
\end{proof}

\begin{rem}
{\bf (a)} As is easily proved by induction, we have
\[
[(T-s)^{-1}T]^{n+1} = (-1)^{n+1}\!\!\sum_{k=n+1}^\infty\binom{k-1}n s^{-k}T^k.
\]
Therefore, for all $s>0$,
\begin{align*}
\sum_{n=0}^\infty\big\|[(T-s)^{-1}T]^{n+1}\big\|
&\le\sum_{n=0}^\infty\sum_{k=n+1}^\infty\binom{k-1}n s^{-k}\|T^k\|\\
&= \sum_{k=1}^\infty\bigg(\sum_{n=0}^{k-1}\binom{k-1}n\bigg)s^{-k}\|T^k\|\\
&= \sum_{k=1}^\infty 2^{k-1}s^{-k}\|T^k\|.
\end{align*}
In particular, we conclude that $\|[(T-s)^{-1}T]^{n+1}\|$ is uniformly bounded for all $n\in\N_0$ and $s\ge s_0>0$. Hence, the term ``$s>0$'' in condition \ref{e:cond99} can be equivalently replaced by ``$s\in (0,s_0)$'' for any $s_0>0$.

\smallskip\noindent
{\bf (b)} In the case where $T$ is injective, the IVP \eqref{e:IVP99} is equivalent to $\dot z = Az$, $z(0) = z_0$, where $A = T^{-1}$, $z = Tx$, and $z_0 = Tx_0$. Condition \ref{e:cond99} can then be rephrased as $\sup\,\{\|\la^{n+1}(\la-A)^{-(n+1)}x_0\| : \la>0,\,n\in\N_0\} < \infty$, which resembles the Hille-Yosida condition for generators of bounded semigroups.

\smallskip\noindent
{\bf (c)} As is seen from the proof of Theorem \ref{t:full}, the necessity of \eqref{e:cond99} for the existence of an $L^\infty$-solution of the IVP does not require $\calX$ to have the Radon-Nikodym property.
\end{rem}

\subsection{Square integrable solutions on compact intervals}
Next, we refrain from the requirement that solutions must be defined on the entire positive time axis $[0,\infty)$. Instead, we shall consider the existence of solutions to \eqref{e:qn} on intervals $[0,\tau]$ for $0<\tau < \infty$.

Recall that if $\calX$ is a Hilbert space, then the Fourier series of a function $F\in L^2([0,\tau];\calX)$ converges (unconditionally) to $F$ in $L^2([0,\tau];\calX)$ (see \cite{ab}). In the proof of the next theorem, we make use of the following lemma which easily follows from the scalar case by exploiting an orthonormal basis of $\calX$.

\begin{lem}
Let $\calX$ be a Hilbert space, $F\in L^2([0,\tau];\calX)$ and let $F(t) = \sum_{k\in\Z}c_ke^{2\pi ikt/\tau}$ be its Fourier series. Then $F\in W^{1,2}([0,\tau];\calX)$ with $F(\tau) = F(0)$ if and only if $(kc_k)_{k\in\Z}\in\ell^2(\Z;\calX)$. In this case, we have $\dot F(t) = \frac 1\tau\sum_{k\neq 0}2\pi ik c_k\cdot e^{2\pi ikt/\tau}$.
\end{lem}

As we saw in Example \ref{ex:malsomalso} (b), there may be distinct solutions of the DAE $\frac d{dt}Tx = x$ on an interval with the same initial value $Tx(0)=Tx_0$. The following proposition shows that the existence of an $L^2$-solution of $\frac d{dt}Tx = x$ on $[0,\tau]$ can be characterized in terms of the difference $Tx(\tau) - Tx(0)$ and that this difference determines the solution uniquely.

\begin{thm}\label{t:l2}
Let $\calX$ be a Hilbert space and let $y\in\calX$. Then there exists a function $x\in L^2([0,\tau];\calX)$ such that $Tx\in W^{1,2}([0,\tau];\calX)$ and
\begin{align}\label{e:BP}
\tfrac d{dt}Tx = x,\quad Tx(\tau)-Tx(0) = y,
\end{align}
if and only if $(\|(I - \frac{2\pi ik}\tau T)^{-1}y\|)_{k\in\Z}\in\ell^2(\Z)$. In this case, the solution $x$ is unique and is given by
\begin{align}\label{e:sol}
x(t) = \tfrac 1\tau\sum_{k\in\Z}(I - \tfrac{2\pi ik}\tau T)^{-1}y\cdot e^{\frac{2\pi ik}\tau t}.
\end{align}
\end{thm}
\begin{proof}
Assume that $(\|(I - \frac{2\pi ik}\tau T)^{-1}y\|)_{k\in\Z}\in\ell^2(\Z)$ and let $x$ be as in \eqref{e:sol}. Then
\begin{align*}
Tx(t)
&= \tfrac 1\tau Ty + \sum_{k\neq 0}\tfrac 1{2\pi ik}(I - \tfrac{2\pi ik}\tau T)^{-1}\tfrac{2\pi ik}\tau Ty\cdot e^{\frac{2\pi ik}\tau t}\\
&= \tfrac 1\tau Ty - \Big(\sum_{k\neq 0}\tfrac 1{2\pi ik}e^{\frac{2\pi ik}\tau t}\Big)y + \sum_{k\neq 0}\tfrac 1{2\pi ik}(I - \tfrac{2\pi ik}\tau T)^{-1}y\cdot e^{\frac{2\pi ik}\tau t}\\
&= \tfrac 1\tau  Ty + (\tfrac t\tau-\tfrac 12)y + \sum_{k\neq 0}\tfrac 1{2\pi ik}(I - \tfrac{2\pi ik}\tau T)^{-1}y\cdot e^{\frac{2\pi ik}\tau t}.
\end{align*}
Note that the coefficients $c_k = \tfrac 1{2\pi ik}(I - \tfrac{2\pi ik}\tau T)^{-1}y$ in the series above satisfy $(\|kc_k\|)\in\ell^2(\Z\backslash\{0\})$. Hence, the function $F$ represented by the series is in $H^1([0,\tau];\calX)$ and satisfies $F(\tau) = F(0)$ so that $Tx(\tau) - Tx(0) = y$. Moreover, the series can be differentiated term-by-term resulting in $\frac d{dt}F$. Hence,
\begin{align*}
\tfrac d{dt}Tx(t) = \tfrac 1\tau y + \sum_{k\neq 0}\tfrac{1}\tau(I - \tfrac{2\pi ik}\tau T)^{-1}y\cdot e^{\frac{2\pi ik}\tau t} = x(t).
\end{align*}
Conversely, let $x\in L^2([0,\tau];\calX)$ with $Tx\in W^{1,2}([0,\tau];\calX)$ such that \eqref{e:BP} holds. Write $x$ as $x(t) = \sum_{k\in\Z}c_ke^{2\pi ikt/\tau}$ with $c_k\in\ell^2(\Z;\calX)$ and define $z(t) := Tx(t) + (\frac 12 - \frac t\tau)y$. Then $z\in W^{1,2}([0,\tau];\calX)$ and
\[
\dot z(t) = \tfrac d{dt}Tx(t) - \tfrac 1\tau y = x(t) - \tfrac 1\tau y = c_0 - \tfrac 1\tau y + \sum_{k\neq 0}c_ke^{\frac{2\pi ik}{\tau}t}.
\]
On the other hand,
\begin{align*}
z(t)
&= \sum_{k\in\Z}Tc_ke^{\frac{2\pi ikt}\tau} + \tfrac 12 y - \Big(\tfrac 12 - \sum_{k\neq 0}\tfrac 1{2\pi ik}e^{\frac{2\pi ik}\tau t}\Big)y\\
&= Tc_0 + \sum_{k\neq 0}(Tc_k + \tfrac 1{2\pi ik}y)e^{\frac{2\pi ik}\tau t}
\end{align*}
and
\[
z(\tau)-z(0) = \big(Tx(\tau) - \tfrac 12y\big) - \big(Tx(0) + \tfrac 12 y\big) = y - y = 0.
\]
Hence,
\[
\dot z(t) = \sum_{k\neq 0}\tfrac{2\pi ik}\tau (Tc_k + \tfrac 1{2\pi ik}y)e^{\frac{2\pi ik}{\tau}t}.
\]
From the identity theorem for Fourier series we conclude that
\[
c_0 = \tfrac 1\tau y\qquad\text{and}\qquad\tau c_k = 2\pi ikTc_k + y,\quad k\in\Z\backslash\{0\}
\]
and thus $c_k = \tfrac 1\tau (I - \tfrac{2\pi ik}\tau T)^{-1}y$ for all $k\in\Z$, so that $(\|(I - \tfrac{2\pi ik}\tau T)^{-1}y\|)_{k\in\Z}\in\ell^2(\Z)$, and $x$ coincides with the function given in \eqref{e:sol}.
\end{proof}

The following two corollaries are immediate consequences of Theorem \ref{t:l2}.

\begin{cor}
Let $\calX$ be a Hilbert space. If $x$ is an $L^2$-solution of the DAE $\frac d{dt}Tx=x$ on $[0,\tau]$ such that $Tx(\tau)=Tx(0)$ \braces{or, equivalently, $\int_0^\tau x(s)\,ds=0$}, then $x=0$.
\end{cor}

\begin{cor}
Let $\calX$ be a Hilbert space. Then there exists an $L^2$-solution of the IVP $\frac d{dt}Tx = x$, $Tx(0) = Tx_0$, on $[0,\tau]$ if and only if there is some $y\in\ran T$ such that $(\|(I-\frac{2\pi ik}\tau T)^{-1}y\|)_{k\in\Z}\in\ell^2(\Z)$ and
\[
Tx_0 = \tfrac 1\tau  Ty - \tfrac 12 y + \sum_{k\neq 0}\tfrac 1{2\pi ik}(I - \tfrac{2\pi ik}\tau T)^{-1}y.
\]
\end{cor}

\begin{rem}
Note that the boundary condition in \eqref{e:BP} implies $y\in\ran T$. If $T$ is injective and we set $A := T^{-1}$, then $(I - \frac{2\pi ik}\tau T)^{-1}y = (A-\frac{2\pi ik}\tau)^{-1}x$, where $Tx = y$. In the case $T = -V$, where $V$ is the Volterra operator from Example \ref{ex:malsomalso}, it can be shown that $x\mapsto ((A-\frac{2\pi ik}\tau)^{-1}x)_{k\in\Z}$ is a bounded operator from $\calX = L^2(0,1)$ to $\ell^2(\Z,\calX)$. Therefore, the $\ell^2$-condition in Theorem \ref{t:l2} is satisfied for any $y\in\ran T$.
\end{rem}

\subsection{Analytic solutions}
Finally, we consider solutions $x$ which are (real-)analytic in a neighborhood of $t=0$. Note that in this case we can replace $\frac d{dt}Tx$ by $T\dot x$.

\begin{prop}\label{p:analytic}
Let $T\in L(\calX)$ be an injective quasi-nilpotent operator. Then the IVP
\begin{align}\label{e:IVP80}
T\dot x = x,\quad x(0)=x_0,
\end{align}
has an analytic solution on some interval $[0,r)$, $r>0$, if and only if
\begin{align}\label{e:cond77}
x_0\in\bigcap_{n=1}^\infty\ran T^n\qquad\text{and}\quad\limsup_{k\to\infty}\tfrac 1k\|T^{-k}x_0\|^{1/k} < \infty.
\end{align}
In this case, the solution of the IVP is unique.

Let $x$ be an analytic solution on $[0,r)$. If $x_0=0$, then $x=0$. If $x_0\neq 0$, then $r\le a/e$, where $a:=\liminf_{n\to\infty}n\|T^n\|^{1/n}$. In particular, if $a=0$, then the only analytic solution of $T\dot x = x$ is the zero-function.
\end{prop}
\begin{proof}
Let $x$ be an analytic solution of $T\dot x = x$ and assume that $x(t) = \sum_{k=0}^\infty x_kt^k$ has radius of convergence $r>0$, where $x_k\in\calX$, $k\in\N$. Then $x_k = \frac{x^{(k)}(0)}{k!}$, and since $Tx^{(k+1)} = x^{(k)}$, we have
$$
Tx_{k+1} = \frac{Tx^{(k+1)}(0)}{(k+1)!} = \frac{x^{(k)}(0)}{(k+1)!} = \frac{x_k}{k+1}.
$$
Hence, $x_0 = k!\cdot T^kx_{k}$, i.e.\ $x_0\in\bigcap_{n=1}^\infty\ran T^n$ and $x_k = (1/k!)T^{-k}x_0$ for all $k\in\N$. This also shows that $x$ is unique as an analytic solution and, in particular, that $x=0$ if $x_0=0$. Moreover,
\[
\tfrac 1r\ge \limsup_{k\to\infty}\tfrac 1{(k!)^{1/k}}\|T^{-k}x_0\|^{1/k}.
\]
By Stirlings formula, we have $k!\sim\sqrt{2\pi k}\left(\frac ke\right)^k$, and thus $(k!)^{1/k}\sim\frac ke$, and \eqref{e:cond77} follows. If $x_0\neq 0$, we may estimate $\|x_0\|\le k!\|T^k\|\|x_{k}\|$.  This implies
$$
\|x_0\|^{1/k}\,\le\,\frac{(k!)^{1/k}}{k/e}\cdot \frac 1e\cdot k\|T^k\|^{1/k}\cdot\|x_k\|^{1/k}.
$$
Choose a subsequence $k_n$ such that $k_n\|T^{k_n}\|^{1/k_n}\to a$. Then the left-hand side converges to $1$, while the limit superior of the right-hand side does not exceed $\frac a{er}$, hence $r\le \tfrac ae$.

Conversely, if \eqref{e:cond77} holds, let $x(t) = \sum_{k=0}^\infty x_kt^k$ with $x_k = (1/k!)T^{-k}x_0$, $k\in\N$. Then \eqref{e:cond77} implies that $x$ is a well defined analytic function on some interval $[0,r)$, $r>0$, and $kTx_k = x_{k-1}$ implies $T\dot x = x$.
\end{proof}

\begin{rem}
{\bf (a)} If the IVP with $x_0\neq 0$ even has a solution which is an entire function, then the Proposition implies that $n\|T^n\|^{1/n}\to\infty$.

\smallskip\noindent
{\bf (b)} The Volterra operator $V$ from Example \ref{ex:malsomalso} satisfies $n!\|V^n\|\to \frac 12$ (see \cite{kershaw}). Therefore, $a=e$ and thus any power series solution of \eqref{e:IVP80} with $x_0\neq 0$ has radius of convergence $r\le 1$.
%
\end{rem}

\noindent{\bf Acknowledgment.} The author would like to thank the reviewers for their careful reading and very helpful comments which led to a significant improvement of the paper. He was funded by the Carl Zeiss Foundation within the project {\em DeepTurb--Deep Learning in and from Turbulence} and by the Free State of Thuringia and the BMBF within the project {\em THInKI--Thüringer Hochschulinitiative für KI im Studium}.

\appendix
\section*{Appendix}
\begin{proof}[Proof of Proposition \rmref{p:proj}]
Let $C_1$ and $C_2$ be two sets of Jordan curves as above surrounding $\sigma$ such that $C_1$ is completely contained in the interior of $C_2$. Then we have (see \eqref{e:TT})
\begin{align*}
P_\sigma^2
&= \frac 1{(2\pi i)^2}\int_{C_1}\int_{C_2}T_zT_s\,dz\,ds = \frac 1{(2\pi i)^2}\int_{C_1}\int_{C_2}\frac{T_s-T_z}{z-s}\,dz\,ds\\
&= \frac 1{(2\pi i)^2}\int_{C_1}T_s\int_{C_2}\frac{1}{z-s}\,dz\,ds + \frac 1{(2\pi i)^2}\int_{C_2}T_z\int_{C_1}\frac{1}{s-z}\,ds\,dz.
\end{align*}
The interior integrals equal $2\pi i$ and $0$, respectively, which implies $P_\sigma^2 = P_\sigma$. Hence, $P_\sigma$ is indeed a projection.

Since $\tfrac{d\tau_\mu}{dz}(z) = (z-\mu)^{-2} = \tau_\mu^2(z)$ and $\mu$ is in the exterior of $C$, we have
$$
0 = \frac 1{2\pi i}\int_C\frac 1{z-\mu}\,dz = -\frac 1{2\pi i}\int_{\tau_\mu(C)}\frac 1\la\,d\la,
$$
hence zero is in the exterior of $\tau_\mu(C)$. Also if $w\in\tau_\mu(\sigma)$, $w = \tau_\mu(s)$ with $s\in\sigma$, then
$$
2\pi i = \int_C\frac 1{z-s}\,dz = \int_{\tau_\mu(C)}\frac{1}{\la((\mu-s)\la-1)}\,d\la = \int_{\tau_\mu(C)}\left(\frac{1}{\la-w} - \frac 1\la\right)\,d\la.
$$
Since zero is in the exterior of $\tau_\mu(C)$, this shows that $w$ (and thus $\sigma$) is in the interior of $\tau_\mu(C)$.

Since $\tau_\mu : \C\setminus\{\mu\}\to\C\setminus\{0\}$ is a homeomorphism, it follows that $\tau_\mu(\sigma)$ is a spectral set for $T_\mu$, cf.\ \eqref{e:corr}. Now,
\begin{align*}
P_\sigma
&= -\frac 1{2\pi i}\int_C\tau_\mu(z)\big(T_\mu - \tau_\mu(z)\big)^{-1}T_\mu\,dz\\
&= -\frac 1{2\pi i}\int_C\tau_\mu(z)\big[I + \tau_\mu(z)\big(T_\mu - \tau_\mu(z)\big)^{-1}\big]\,dz\\
&= -\frac 1{2\pi i}\int_C\tau_\mu(z)^2\big(T_\mu - \tau_\mu(z)\big)^{-1}\,dz = 
-\frac 1{2\pi i}\int_{\tau_\mu(C)}\big(T_\mu - \la\big)^{-1}\,d\la.
\end{align*}
This proves the equation \eqref{e:Psigma} and also that $\tau_\mu(C)$ is positively oriented since otherwise $-P_\sigma$ would be a projection.
\end{proof}

\bigskip
\section*{Author affiliation}
\vspace*{-.4cm}
\end{document}